\documentclass[10pt,a4paper]{amsart}

\theoremstyle{plain}
\newtheorem{theorem}{Theorem}[section]
\newtheorem{lemma}[theorem]{Lemma}
\newtheorem{proposition}[theorem]{Proposition}

\theoremstyle{definition}

\theoremstyle{remark}
\newtheorem{remark}[theorem]{Remark}

\def\bin #1#2 {\left( \matrix { #1 \cr #2 \cr } \right) }

\begin{document}

\title[The genus of  curves in $\mathbb P^4$ and $\mathbb P^5$ not contained in quadrics]
{The genus of  curves in $\mathbb P^4$ and $\mathbb P^5$ not
contained in quadrics}

\author{Vincenzo Di Gennaro }
\address{Universit\`a di Roma \lq\lq Tor Vergata\rq\rq, Dipartimento di Matematica,
Via della Ricerca Scientifica, 00133 Roma, Italy.}
\email{digennar@mat.uniroma2.it}

\abstract A classical problem in the theory of projective curves is
the classification of all their possible genera in terms of the
degree  and the dimension of the space where they are embedded.
Fixed integers $r,d,s$, Castelnuovo-Halphen's theory states a sharp
upper bound for the genus of a non-degenerate, reduced and
irreducible curve of degree $d$ in $\mathbb P^r$, under the
condition of being not contained in a surface of degree $<s$. This
theory can be generalized in several ways. For instance, fixed
integers $r,d,k$, one may ask for the maximal genus of a curve of
degree $d$ in $\mathbb P^r$, not contained in a hypersurface of a
degree $<k$. In the present paper we examine the genus of curves $C$
of degree $d$ in $\mathbb P^r$ not contained in quadrics (i.e.
$h^0(\mathbb P^r, \mathcal I_C(2))=0$). When $r=4$ and $r=5$, and
$d\gg0$, we exhibit a sharp upper bound for the genus. For certain
values of $r\geq 7$, we are able to determine a sharp bound except
for a constant term, and the argument applies also to curves not
contained in cubics.

Quadric and cubic hypersurfaces. Veronese surface. Projection of a
rational normal scroll surface. Maximal rank.


\endabstract
\maketitle

\section{Introduction}
A classical problem in the theory of projective curves is the
classification of all their possible genera in terms of the degree
and the dimension of the space where they are embedded
\cite{Hartshorne}, \cite{EH}, \cite{CCD}. Fixed integers $r,d,s$,
Castelnuovo-Halphen's theory states a sharp upper bound for the
genus of a non-degenerate, reduced and irreducible curve of degree
$d$ in $\mathbb P^r$, under the condition of being not contained in
a surface of degree $<s$ \cite{EH}, \cite{GP1}, \cite{CCD}. This
theory can be generalized imposing flag conditions \cite{CCD2}. For
instance, fixed integers $r,d,k$, one may ask for the maximal genus
of a curve of degree $d$ in $\mathbb P^r$, not contained in a
hypersurface of a degree $<k$. As far as we know, in this case there
are only rough estimates \cite{NV}, \cite[p. 726, (2.10)]{CCD2}. In
the present paper we examine the case of curves in $\mathbb P^4$ and
$\mathbb P^5$ not contained in quadrics. Our main results are the
following Theorem \ref{bound} and Theorem \ref{boundP5}.

\begin{theorem}\label{bound} Let $C\subset \mathbb P^4$ be a
reduced and irreducible complex curve, of degree $d$ and arithmetic
genus $p_a(C)$, not contained in quadrics (i.e. $h^0(\mathbb
P^4,\mathcal I_C(2))=0$).

\smallskip
$\bullet$\,\,\, If $d> 16$, then
\begin{equation*}\label{G4}
p_a(C)\leq \frac{1}{8}d(d-6)+1.
\end{equation*}
The bound is sharp, and $p_a(C)=\frac{1}{8}d(d-6)+1$ if and only if
$d$ is even, and $C$ is contained in the isomorphic projection in
$\mathbb P^4$ of the Veronese surface.{\footnote{\,The bound is the
genus of a plane curve of degree $d/2$. It should be compared with
the sharp lower bound  $K^2_S\geq -d(d-6)$ and $\chi(\mathcal
O_S)\geq -\frac{1}{8}d(d-6)$, for a smooth projective surface $S$ of
degree $d$ \cite{DGFr}, \cite{DG}.}}

\smallskip
$\bullet$\,\,\,If $d> 143$, and either $d$ is odd or $C$ is
arithmetically Cohen-Macaulay (a.C.M. for short), then
\begin{equation}\label{G44}
p_a(C)\leq
\frac{d^2}{10}-\frac{d}{2}-\frac{1}{10}(\epsilon-4)(\epsilon+1)+\binom{\epsilon}{4}+1,
\end{equation}
where $\epsilon$ is defined by dividing $d-1=5m+\epsilon$, $0\leq
\epsilon \leq 4$. The bound is sharp. Every extremal curve is
a.C.M., and it is contained in a flag like $S\subset T\subset
\mathbb P^4$, where $S$ is a unique surface of degree $5$, and $T$
is a cubic threefold. Moreover, the surface $S$ has sectional genus
$\pi=1$.
\end{theorem}

When $d$ is even, in contrast  with the classical case, the extremal
curves are not a.C.M.. Moreover,  the asymptotic behaviour of the
bound  is different, depending on whether $d$ is even or odd. As for
the proof of Theorem \ref{bound}, the first part, when $d\gg 0$,
easily follows combining the geometry of the Veronese surface (see
Lemma \ref{veronese} below), with \cite[p. 117, Theorem (3.22)]{EH}.
In order to establish the bound under the hypothesis $d> 16$, we
need to refine the analysis of the Hilbert function of the general
hyperplane section of $C$, similarly as in \cite[p. 74-75]{DGFr}. It
is a rather long numerical argument, relying on \cite{CCD}. We
relegate it to an appendix at the end of the paper (Section $7$). We
will not insist on this analysis in the other cases (i.e. in second
part of Theorem \ref{bound}, and in next Theorem \ref{boundP5},
Proposition \ref{qboundPr}, and Proposition \ref{cboundPr}). We will
be content with coarser hypotheses on $d$, which can be obtained
without too much effort, simply using \cite[loc. cit.]{EH},
\cite[Corollary 3.11]{DGF}, and \cite[Main Theorem]{CCD}. However,
we think our assumptions on $d$, in the second part of Theorem
\ref{bound} and in next Theorem \ref{boundP5},  can be significantly
improved. We hope to return on this question in a forthcoming paper.
When $d$ is odd, or $C$ is a.C.M., the proof of Theorem \ref{bound}
is more involved, and we need \cite[Theorem 1]{Noma} in order to
extend certain results of \cite{DGF} in the range $r=4$ and $s=5$.


\begin{theorem}\label{boundP5} Let $C\subset \mathbb P^5$ be a
reduced and irreducible complex curve, of degree $d$ and arithmetic
genus $p_a(C)$, not contained in quadrics. Assume $d>215$, and
divide $d-1=6m+\epsilon$, $0\leq\epsilon\leq 5$. Then:
$$
p_a(C)\leq 6\binom{m}{2}+m\epsilon.
$$
The bound is sharp{\,\footnote{\, This is the Castelnuovo's bound
for curves of degree $d$ in $\mathbb P^7$ (see Section $2$, $(v)$,
below).}}. Every extremal curve is not a.C.M., and it is contained
in a flag like $S\subset T\subset \mathbb P^5$, where $S$ is a
unique surface  of degree $6$, and $T$ is a cubic hypersurface of
$\mathbb P^5$. Moreover, $S$ has sectional genus $\pi=0$, and
arithmetic genus $p_a(S)=0$.
\end{theorem}

The analysis of the arithmetic genus of a surface of degree $6$ in
$\mathbb P^5$ appearing in \cite[Corollary 3.6]{DGF}, combined with
\cite[loc. cit.]{EH}, enables us to state the bound in Theorem
\ref{boundP5}. As for the sharpness, we simply project a general
Castelnuovo's curve in $\mathbb P^7$, which is not contained in
quadrics because so is the general projection of a smooth rational
normal scroll surface \cite[Theorem 2, Lemma 3.1]{BE}.

\medskip
In the case of curves contained in $\mathbb P^r$, for certain values
of $r\geq 7$, we are able to compute the sharp bound except for  a
constant term. We obtain similar results also in the case of cubics.
In fact, we prove the following Proposition \ref{qboundPr} and
Proposition \ref{cboundPr}. For the definition of the number
$d_1(r,s)$, appearing in the claims, see Section $(2)$, $(vi)$, and
(\ref{es2}) below.

\begin{proposition}\label{qboundPr}
Fix an integer $r\geq 7$, not divisible by $3$, and divide
$$
\binom{r+2}{2}=3h+k, \quad 0\leq k\leq 2.
$$
Set: \footnote{{\, The number $s$ is the minimal integer such that
$\binom{r+2}{2}-3(s+1)$ is $\leq 0$ (e.g. $(r,s)=(7,11),\, (8,14),\,
(10, 21)$). The number $r$ is not divisible by $3$  if and only if
$k\neq 1$.}}
$$
s=\begin{cases} h-1 \quad {\text{if $k=0$}}\\h \quad {\text{if
$k=2$}}.
\end{cases}
$$
Let $C\subset \mathbb P^r$ be a reduced and irreducible complex
curve, of degree $d$ and maximal arithmetic genus $p_a(C)$ with
respect to the conditions of being of degree $d$ and not contained
in a quadric hypersurface. Assume $d>d_1(r,s)$, and divide
$d-1=ms+\epsilon$, $0\leq\epsilon\leq s-1$. Then:
\begin{equation}\label{bqs}
p_a(C)=\frac{d^2}{2s}-\frac{d}{2s}(s+2)+O(1), \quad {\text{with
$0<O(1)\leq \frac{s^3}{r-2}$}.}
\end{equation}
Moreover, $C$ is not a.C.M., and it is contained in a flag like
$S\subset T\subset \mathbb P^r$, where $S$ is a unique surface  of
degree $s$, and $T$ is a cubic hypersurface of $\mathbb P^r$. The
surface $S$ has  sectional genus $\pi=0$.
\end{proposition}

\begin{proposition}\label{cboundPr}
Fix an integer $r\geq 9$, and assume that the number
$$
s:=\frac{1}{6}\left[\binom{r+3}{3}-4\right]
$$
is an integer.{\footnote{\, This is equivalent to say that the class
$[r]$ of $r$ modulo $36$ is one of the following classes: $[1]$,
$[2]$, $[9]$, $[10]$, $[11]$, $[18]$, $[19]$, $[27]$, $[29]$ (e.g.
$(r,s)=(9,36), (10,47), (11,60), (18,221), (19, 256)$).}} Let
$C\subset \mathbb P^r$ be a reduced and irreducible complex curve,
of degree $d$ and maximal arithmetic genus $p_a(C)$ with respect to
the conditions of being of degree $d$ and not contained in a cubic
hypersurface. Assume $d>d_1(r,s)$, and divide $d-1=ms+\epsilon$,
$0\leq\epsilon\leq s-1$. Then:
\begin{equation}\label{bqs1}
p_a(C)=\frac{d^2}{2s}-\frac{d}{2s}(s+2)+O(1), \quad {\text{with
$0<O(1)\leq \frac{s^3}{r-2}$}.}
\end{equation}
Moreover, $C$ is not a.C.M., and it is contained in a flag like
$S\subset T\subset \mathbb P^r$, where $S$ is a unique surface  of
degree $s$, and $T$ is a quartic hypersurface of $\mathbb P^r$. The
surface $S$ has sectional genus $\pi=0$.
\end{proposition}

The proof follows the same line of Theorem \ref{boundP5}. However,
this analysis, when $r\geq 7$, leads to examine surfaces whose
degree is out of the range considered in \cite[Theorem 2.2]{DGF}. We
partly overcome this difficulty, using an estimate appearing in
\cite[Lemma]{PAMS}. This explains why we are not able to determine
the sharp bound (i.e. the constant $O(1)$).

\bigskip
\section{Notations and preliminary remarks.}

$\bullet$ $(i)$ For a projective subscheme $V\subseteq \mathbb P^N$
we will denote by $\mathcal I_V=\mathcal I_{V,\mathbb P^N}$ its
ideal sheaf in $\mathbb P^N$, and by $M(V):=\oplus_{i\in\mathbb
Z}H^1(V,\mathcal I_V(i))$ the Hartshorne-Rao module of $V$. We will
denote by $h_V$ the Hilbert function of $V$, and by $\Delta h_V$ the
first difference of $h_V$, i.e. $\Delta h_V(i)=h_V(i)-h_V(i-1)$.
Observe that
\begin{equation}\label{basic}
h_V(i)=h^0(\mathbb P^N,\mathcal O_{\mathbb P^N}(i))-h^0(\mathbb
P^N,\mathcal I_V(i))\leq h^0(V,\mathcal O_V(i)).
\end{equation}
We denote by $p_a(V)$ the arithmetic genus of $V$. We say that $V$
is {\it arithmetically Cohen-Macaulay} (shortly a.C.M.) if all the
restriction maps $H^0(\mathbb P^N, \mathcal O_{\mathbb P^N}(i))\to
H^0(V,\mathcal O_V(i))$ ($i\in\mathbb Z$) are surjective, and
$H^j(V,\mathcal O_V(i))=0$ for all $i\in\mathbb Z$ and $1\leq j\leq
\dim V-1$.  If $V$ is one-dimensional of degree $d$, and $V'$
denotes its general hyperplane section, then \cite[p. 83-84]{EH}:
\begin{equation}\label{basic2}
\Delta h_V(i)\geq h_{V'}(i) \,\,{\text {for every $i$}},
\end{equation}
and
\begin{equation}\label{basic2.1}
p_a(V)\leq \sum_{i=1}^{+\infty} d-h_{V'}(i).
\end{equation}
Moreover, the equality occurs in (\ref{basic2}) (resp. in
(\ref{basic2.1})) if and only if $V$ is a.C.M.. If $V$ is integral
(i.e. reduced and irreducible) and one-dimensional, we also have
\cite[p. 86-87, Corollary (3.5) and (3.6)]{EH}:
\begin{equation}\label{basic2.2}
h_{V'}(i+j)\geq \min\{d,\, h_{V'}(i)+h_{V'}(j)-1\} \,\,{\text {for
every $i,j$}},
\end{equation}
and
\begin{equation}\label{basic2.3}
h_{V'}(i)\geq \min\{d,\, i(r-1)+1\} \,\,{\text {for every $i$}}.
\end{equation}
If $V$ is integral, and $\dim V\geq 2$, then $V$ is a.C.M. if and
only if its general hyperplane section is.

\medskip
$\bullet$ $(ii)$ Let $\Sigma\subset \mathbb P^{r-1}$ be an integral
curve of degree $s$ in $\mathbb P^{r-1}$, $r\geq 4$, and
$\Sigma'\subset \mathbb P^{r-2}$  its general hyperplane section.
For every $i$ we have: $ h^0(\Sigma, \mathcal
O_{\Sigma}(i))=1-\pi+is+ h^1(\Sigma, \mathcal O_{\Sigma}(i)),$ where
$\pi$ denotes the arithmetic genus of $\Sigma$. From the exact
sequence $ 0\to\mathcal O_{\Sigma}(i-1)\to \mathcal O_{\Sigma}(i)\to
\mathcal O_{\Sigma'}(i)\to 0$ we get the exact sequence $ 0\to
H^{0}(\Sigma,\mathcal O_{\Sigma}(i-1))\to H^{0}(\Sigma,\mathcal
O_{\Sigma}(i))\to H^{0}(\Sigma',\mathcal O_{\Sigma'}(i)) \to
H^{1}(\Sigma,\mathcal O_{\Sigma}(i-1))\to H^{1}(\Sigma,\mathcal
O_{\Sigma}(i))\to 0. $ We deduce that $h^{1}(\Sigma,\mathcal
O_{\Sigma}(i))\leq h^{1}(\Sigma,\mathcal O_{\Sigma}(i-1)).$ In
particular, when $i\geq 1$, we have $ h^{1}(\Sigma,\mathcal
O_{\Sigma}(i))\leq \pi=h^{1}(\Sigma,\mathcal O_{\Sigma}).$ Observe
that if, for some $i\geq 1$, one has $h^{1}(\Sigma,\mathcal
O_{\Sigma}(i))=\pi$, then $\pi=0$. In fact, in this case, we have
$h^{1}(\Sigma,\mathcal O_{\Sigma}(1))= \pi$, hence $h^0(\Sigma,
\mathcal O_{\Sigma}(1))=1+s$. This means that $\Sigma$ is contained
in $\mathbb P^{s}$ as a non-degenerate curve of degree $s$, so
$\pi=0$. Similarly, if for some $i\geq 1$, one has
$h^{1}(\Sigma,\mathcal O_{\Sigma}(i))=\pi-1$, then $\pi=1$. In fact,
in this case, we have $h^{1}(\Sigma,\mathcal O_{\Sigma}(1))= \pi-1$,
hence $h^0(\Sigma, \mathcal O_{\Sigma}(1))=s$. This means that
$\Sigma$ is contained in $\mathbb P^{s-1}$ as a non-degenerate curve
of degree $s$. By Castelnuovo's bound, it follows that  $\pi=1$. In
conclusion, since $ h^0(\Sigma, \mathcal O_{\Sigma}(i))=
1-\pi+is+h^1(\Sigma, \mathcal O_{\Sigma}(i))$, for every $i\geq 1$
we have:
\begin{equation}\label{basic3}
h^0(\Sigma, \mathcal O_{\Sigma}(i))\leq
\begin{cases} 1+is \quad {\text{if $\pi\geq 0,$}}\\
is \quad {\text{if $\pi\geq 1,$}}\\
-1+is \quad {\text{if $\pi\geq 2$}}.
\end{cases}
\end{equation}

\medskip
$\bullet$ $(iii)$ Let $S\subset \mathbb P^r$ ($r\geq 4$) be an
integral surface of degree $s$, and  $\Sigma$ its general hyperplane
section. From the exact sequence $ 0\to\mathcal O_{S}(i-1)\to
\mathcal O_{S}(i)\to \mathcal O_{\Sigma}(i)\to 0$ we get, for every
$i\geq 1$, $ h^{0}(S,\mathcal O_{S}(i))\leq
\sum_{j=0}^{i}h^{0}(\Sigma,\mathcal O_{\Sigma}(j)).$ Hence, for
every $i\geq 1$, we have (compare with (\ref{basic})) $ h^0(\mathbb
P^r,\mathcal I_S(i))\geq
\binom{r+i}{i}-\sum_{j=0}^{i}h^{0}(\Sigma,\mathcal O_{\Sigma}(j)).$
Combining with (\ref{basic3}), it follows that:
\begin{equation}\label{basic4}
h^0(\mathbb P^r,\mathcal I_S(i))\geq
\begin{cases} \binom{r+i}{i}-\left[i+1+\binom{i+1}{2}s\right] \quad {\text{if $\pi\geq 0,$}}\\
\binom{r+i}{i}-\left[1+\binom{i+1}{2}s\right] \quad {\text{if $\pi\geq 1,$}}\\
\binom{r+i}{i}-\left[1-i+\binom{i+1}{2}s\right] \quad {\text{if
$\pi\geq 2$}}.
\end{cases}
\end{equation}

\medskip
$\bullet$ $(iv)$ Let $S\subset \mathbb P^{s+1}$ be a non-degenerate
smooth rational normal scroll surface, of minimal degree $s$. Fix an
integer $k\geq 1$. Since $S$ is a.C.M., we have $h^1(S,\mathcal
O_S(k))=0$. Moreover, one has $K_{S}=-2\Sigma+(s-2)W$, where
$\Sigma$ is the general hyperplane section, and $W$ the ruling.
Therefore, one has $h^2(S,\mathcal O_S(k))=h^0(S, K_S-k\Sigma)=0$.
By Riemann-Roch Theorem it follows that: $ h^0(S,\mathcal
O_S(k))=1+\frac{1}{2}k\Sigma\cdot(k\Sigma-K_S)=k+1+\binom{k+1}{2}s.
$ In particular:
\begin{equation}\label{scroll}
h^0(S,\mathcal O_S(2))=3(1+s),\quad {\text{and}}\quad h^0(S,\mathcal
O_S(3))=4+6s.
\end{equation}

\medskip
$\bullet$ $(v)$ Fix integers $s\geq 2$ and $d\geq s+1$, and divide
$d-1=ms+\epsilon$, $0\leq \epsilon\leq s-1$. The number
$$
G(s+1;d):=\binom{m}{2}s+m\epsilon
$$
is the celebrated Castelnuovo's bound for the genus of a
non-degenerate integral curve of degree $d$ in $\mathbb P^{s+1}$
\cite[p. 87, Theorem (3.7)]{EH}. Observe that
\begin{equation}\label{b1}
G(s+1;d)=\frac{d^2}{2s}+\frac{d}{2s}(-s-2)+\frac{1+\epsilon}{2s}(s+1-\epsilon)\leq
\frac{d^2}{2s}.
\end{equation}

\medskip
$\bullet$ $(vi)$ Fix integers $r$, $d$ and $s$, with $s\geq r-1\geq
2$. Denote by $G(r ; d, s)$ the maximal arithmetic genus for an
integral non-degenerate projective curve $C\subset\mathbb P^r$ of
degree $d$, not contained in any surface of degree $< s$ \cite{CCD}.
Put
$$
d_0(r,s):=\frac{2s}{r-2}\prod_{i=1}^{r-2}[(r-1)!s]^{\frac{1}{r-1-i}}.
$$
By \cite[Main Theorem]{CCD} and  \cite[Lemma]{PAMS}, one knows that,
for $d>d_0(r,s)$, the number $G(r;d,s)$ has the following form:
\begin{equation}\label{b2}
G(r;d,s)=\frac{d^2}{2s}+\frac{d}{2s}[2G(r-1;s)-2-s]+R,\quad
{\text{with}}\quad  |R|\leq \frac{s^3}{r-2}.
\end{equation}
When $r=4$ and $s=6$, this holds true also for $d>143$, and when
$r=5$ and $s=7$, this holds true also for $d>179$ \cite[Theorem
(3.22), p. 117]{EH}. Moreover, taking into account (\ref{b2}), and
that $G(r-1;s+1)\leq \frac{(s+1)^2}{2(r-2)}$ (compare with
(\ref{b1})), an elementary computation, which we omit, shows that,
for
\begin{equation}\label{es2}
d>d_1(r,s):=\max\left\{d_0(r,s),\, \frac{4s}{r-2}(s+1)^3\right\},
\end{equation}
one has
\begin{equation}\label{es1}
G(r;d,s+1)<G(s+1;d).
\end{equation}

\medskip
$\bullet$ $(vii)$ Let $S\subset \mathbb P^r$ be an integral
non-degenerate projective surface of degree $s\geq r-1\geq 2$.
Denote by $\sigma$ the integer part of the number
$$
(s-r+2)\left(\frac{s^2}{2(r-2)}+1\right)+1.
$$
By \cite[Lemma 3.8]{DGF} we know that
\begin{equation}\label{reg}
h^1(\mathbb P^r,\mathcal I_S(i))=0 \quad{\text{for every $i\geq
\sigma$.}}
\end{equation}
Now, suppose there exists an a.C.M. curve $C$ on $S$ of degree $d$,
with $d-1=ms+\epsilon$, $0\leq \epsilon\leq s-1$, and $m\geq
\sigma$. Let $\Sigma$ and $\Gamma$ be the general hyperplane
sections of $S$ and $C$. By Bezout's theorem we have $h_S(i)=h_C(i)$
and $h_{\Sigma}(i)=h_{\Gamma}(i)$ for every $i\leq m$. Since $C$ is
a.C.M.,  we have $\Delta h_C(i)=h_{\Gamma}(i)$ for every $i$
(compare with  (\ref{basic2}) and three lines below). It follows
that $\Delta h_S(i)=h_{\Sigma}(i)$ for every $i\leq m$. Therefore,
since
\begin{equation}\label{mui}
\mu_i:=\Delta h_S(i)-h_{\Sigma}(i)=\dim_{\mathbb
C}\left[\ker\left(H^1(\mathbb P^r,\mathcal I_S(i-1))\to H^1(\mathbb
P^r,\mathcal I_S(i))\right)\right]
\end{equation}
(see \cite[Lemma 3.4 and Remark 3.5]{DGF}), we get $H^1(\mathbb
P^r,\mathcal I_S(i-1))\subseteq H^1(\mathbb P^r,\mathcal I_S(i))$
for every $i\leq m$. Combining with (\ref{reg}), it follows that
$M(S)=0$. Summing up: {\it if a surface $S$ contains an a.C.M. curve
of degree $d$ with $m\geq \sigma$, then $H^1(\mathbb P^r,\mathcal
I_S(i))=0$ for every $i\in\mathbb Z$} (compare with \cite[Remark
3.10, $(i)$]{DGF}). Notice that the condition $m\geq \sigma$ is
satisfied when $r=5$ and $s=6$ and $d>179$, or when $d>d_1(r,s)$.

\bigskip
\section{The proof of Theorem \ref{bound}.}

We start with the proof of the first part of  Theorem \ref{bound}.
We need the following Lemma \ref{veronese}, which is certainly well
known. We prove it for lack of a suitable reference.

\begin{lemma}\label{veronese} {\it Let $V\subset \mathbb P^5$ be the Veronese surface.
Let ${\text{Sec}}(V)$ be the secant variety of $V$, $x\in \mathbb
P^5\backslash {\text{Sec}}(V)$ be a point. Then the projection in
$\mathbb P^4$ of $V$ from $x$ is a surface of degree $4$, not
contained in quadrics. Conversely, every integral surface of degree
$4$ in $\mathbb P^4$ not contained in quadrics is the projection of
$V$ from a point $x\in \mathbb P^5\backslash {\text{Sec}}(V)$.}
\end{lemma}

\begin{proof} Let $V'$ be the projection in
$\mathbb P^4$ of $V$ from a point $x\in \mathbb P^5\backslash
{\text{Sec}}(V)$. Suppose there exists  a quadric $Q$ in $\mathbb
P^4$ containing $V'$. If $Q$ were smooth, then, by Severi's theorem,
$V$ would be a complete intersection of two quadrics. This implies
that the sectional genus of $V$ is $1$, in contrast with the fact
that it is $0$. If $Q$ were singular, then $Q$ would be a cone. Let
$H\subset \mathbb P^4$ be a general hyperplane, passing from the
vertex of the cone (if it is a point). Then $H\cap V$ is an integral
curve of degree $4$ of a quadric cone in $\mathbb P^3$. Therefore,
the genus of $H\cap V$ would be $1$, a contradiction, because it is
$0$.

\smallskip
Conversely, let $S$ be a surface of degree $4$ in $\mathbb P^4$ not
contained in quadrics. Let $\pi$ be the sectional genus of $S$. By
Castelnuovo's bound, we have $\pi=0$ either $\pi=1$. If $\pi=1$,
then the general hyperplane section $H\cap S$ of $S$ is a
Castelnuovo's curve in $\mathbb P^3$. This curve is contained in a
quadric, which lifts to a quadric containing $S$ because $H\cap S$
is a.C.M., hence also $S$ is. Therefore, $\pi=0$. In this case, by
\cite[Lemma 7, p. 411]{SD}, we know that $S$ is the projection of
$V$ from some point $x\in \mathbb P^5\backslash V$. Now we observe
that $x\notin {\text{Sec}}(V)\backslash V$, otherwise $S$ is
contained in a quadric \cite[Remark 2.1, p. 60-61]{MLN}.
\end{proof}

We are in position to prove Theorem \ref{bound}, first part.

Let $C\subset \mathbb P^4$ be a curve of degree $d$ not contained in
quadrics. Since a surface of degree $3$ in $\mathbb P^4$ is
contained in a quadric (Section 2, (\ref{basic4})), in view of
previous Lemma \ref{veronese}, it suffices to prove that {\it if $C$
is not contained in a surface of degree $<5$, then
$p_a(C)<\frac{1}{8}d(d-6)+1$}. To this purpose, set
$$G:=\frac{1}{10}d^2-\frac{3}{10}d+\frac{1}{5}+\frac{1}{10}v-\frac{1}{10}v^2+w,
$$
where $v$ is defined by dividing  $d-1=5n+v$, $0\leq v\leq 4$, and
$w:=\max\{0, [\frac{v}{2}]\}$. By \cite[Theorem (3.22), p. 117]{EH}
(with the notation of \cite{EH} one has $G=\pi_{2}(d,4)$), we know
that, if $C$ is not contained in a surface of degree $<5$, and
$d>143$,  then $p_a(C)\leq G$. An elementary computation shows that
$
G<\frac{1}{8}d(d-6)+1
$
for $d>18$. This concludes the proof of the first part of Theorem
\ref{bound}, when $d>143$. We may examine the remaining cases
$16<d\leq 143$ in a similar manner as in \cite[p. 74-75]{DGFr}. For
details, we refer to the Appendix (see Section 7 below).

\medskip
Now we are going to prove Theorem \ref{bound}, second part.

First, we notice that if $d$ is odd, or  $C$ is a.C.M., and $C$ is
not contained in quadrics, then $C$ is not contained in surfaces of
degree $<5$. In fact,  by (\ref{basic4}), every surface in $\mathbb
P^4$ of degree $3$ is contained in a quadric. This is true also for
surfaces of degree $4$, except for an isomorphic projection $S$ of
the Veronese surface (Lemma \ref{veronese}). But on the Veronese
surface every curve has degree even. And every curve on $S$ cannot
be a.C.M. (otherwise,  from the natural sequence: $ 0\to\mathcal
I_S\to \mathcal I_C\to \mathcal O_S(-C)\to 0,$ we would get the
exact sequence:
$
H^0(S, \mathcal O_S(H_S-C))\to H^1(\mathbb P^4, \mathcal I_S(1))\to
H^1(\mathbb P^4, \mathcal I_C(1));
$
this is impossible, because $H^0(S, \mathcal O_S(H_S-C))=0$ and
$H^1(\mathbb P^4, \mathcal I_S(1))\neq 0$).

\medskip
On the other hand, if $C$ is not contained in a surface of degree
$<6$, and $d>143$, then (compare with Section 2, $(vi)$, and
(\ref{b2}))
$$
p_a(C)\leq G(4;d,6)\leq \frac{d^2}{12}+108.
$$
By an elementary comparison, it follows  that, for $d>143$, $p_a(C)$
is strictly less than the bound (\ref{G44}) appearing in our claim.
Therefore, in order to prove the second part of Theorem \ref{bound},
we may assume that $C$ is contained in an integral surface $S\subset
\mathbb P^4$ of degree $5$, not contained in quadrics. Observe that,
by Bezout's theorem, such a surface $S$ is unique, and is contained
in a cubic hypersurface of degree $3$ by (\ref{basic4}).

\smallskip Let $\Sigma\subset \mathbb P^3$ be the general
hyperplane section of $S$. Since $\deg \Sigma=5$, by Castelnuovo's
bound, the arithmetic genus $\pi$ of $\Sigma$ satisfies the
condition $0\leq \pi\leq 2$. The sectional genus $\pi$ cannot be
$2$, otherwise $\Sigma$ should be a Castelnuovo's curve, hence
a.C.M., and contained in a quadric. This would imply that also $S$
is a.C.M., and contained in a quadric (since $h^1(\mathbb
P^4,\mathcal I_{S}(1))=0$, the restriction map $H^0(\mathbb
P^4,\mathcal I_{S}(2))\to H^0(\mathbb P^3,\mathcal I_{\Sigma,\mathbb
P^3}(2))$, induced by the exact sequence $0\to \mathcal I_S(1)\to
\mathcal I_S(2)\to \mathcal I_{\Sigma,\mathbb P^3}(2)\to 0$, is
onto). Therefore, we have $0\leq \pi\leq 1$. Since $\Sigma$ is
non-degenerate, from the exact sequence
$
0\to \mathcal I_{\Sigma,\mathbb P^3}(1)\to \mathcal O_{\mathbb
P^3}(1)\to \mathcal O_{\Sigma}(1)\to 0
$
it follows that $h^1(\mathbb P^3,\mathcal I_{\Sigma,\mathbb
P^3}(1))=2-\pi>0$ (observe that $h^1(\Sigma,\mathcal
O_{\Sigma}(i))=0$ for every $i\geq 1$, because $0\leq \pi\leq 1$).
Moreover, by \cite[Theorem 1]{Noma}, we have
\begin{equation}\label{noma}
h^1(\mathbb P^3,\mathcal I_{\Sigma,\mathbb
P^3}(i-\pi))=0\quad{\text{for every $i\geq 3$}}.
\end{equation}
Let $\Sigma'\subset \mathbb P^2$ be the general plane section of
$\Sigma$. By (\ref{basic2.3}) we have $h_{\Sigma'}(i)=5$ for every
$i\geq 2$. Hence, $h^1(\mathbb P^2,\mathcal I_{\Sigma',\mathbb
P^2}(i))=0$ for every $i\geq 2$, and from the exact sequence $ 0\to
\mathcal I_{\Sigma,\mathbb P^3}(i-1)\to \mathcal I_{\Sigma,\mathbb
P^3}(i)\to \mathcal I_{\Sigma',\mathbb P^2}(i)\to 0$ it follows that
$h^1(\mathbb P^3,\mathcal I_{\Sigma,\mathbb P^3}(2))\leq h^1(\mathbb
P^3,\mathcal I_{\Sigma,\mathbb P^3}(1))$. Summing up, for every
$i\geq 1$, we have:
$$h^1(\mathbb P^3,
\mathcal I_{\Sigma}(i))\leq \max\{0,\, 3-\pi-i\}+\mu(i),$$ with
$\mu(i)=1$ if $\pi=0$ and $i=2$, and  $\mu(i)=0$ otherwise. Now,
divide $d-1=5m+\epsilon$, $0\leq \epsilon\leq 4$, and set (recall
that $0\leq \pi\leq 1$):
$$
h_{d,\pi}(i)=1-\pi+5i-\max\{0,\, 3-\pi-i\}-\mu(i)
$$
for $1\leq i\leq m$, and, for $i\geq m+1$, $ h_{d,\pi}(i)=d$, except
for the case  $\pi=1$, $\epsilon =4$, $i=m+1$, in which we set $
h_{d,\pi}(m+1)=d-1.$ If we denote by $\Gamma$ the general hyperplane
section of $C$, the same analysis appearing in the proof of
\cite[Lemma 3.3]{DGF}, shows that $ h_{\Gamma}(i)\geq h_{d,\pi}(i)$
for every $i\geq 1$. By (\ref{basic2.1}) it follows that
\begin{equation}\label{acm}
p_a(C)\leq \sum_{i=1}^{+\infty} d-h_{\Gamma}(i)\leq
\sum_{i=1}^{+\infty} d-h_{d,\pi}(i)=:G_{d,\pi}.
\end{equation}
Now, observe that
$$
G_{d,0}=5\binom{m}{2}+m\epsilon+4,\quad {\text{and}}\quad
G_{d,1}=5\binom{m}{2}+m(\epsilon+1)+1+\binom{\epsilon}{4}.
$$
Therefore, since $G_{d,0}<G_{d,1}$, and taking into account that
$G_{d,1}$ is exactly the bound appearing in (\ref{G44}), in order to
complete the proof we only have to exhibit examples with
$p_a(C)=G_{d,1}$ (in this case, by (\ref{acm}) and Section 2, $(i)$,
we know that $C$ is a.C.M.).

\smallskip To this purpose, fix an elliptic curve $\Sigma\subset \mathbb P^3$
of degree $5$ (compare with \cite[p. 101-103]{DGF}). By (\ref{noma})
we know that $h^1(\mathbb P^3,\mathcal I_{\Sigma}(2))=0$. It follows
that $\Sigma$ is not contained in quadrics, because $H^0(\mathbb
P^3,\mathcal O_{\mathbb P^3}(2))$ and $H^0(\Sigma,\mathcal
O_{\Sigma}(2))$ have the same dimension, $10$. Let
$S=C(\Sigma)\subset\mathbb P^4$ be the cone over $\Sigma$. Let
$D\subset \Sigma$ be a subset formed by $4-\epsilon$ points. Let
$C(D)\subset C(\Sigma)$ be the cone over $D$. Let $\mu\geq 3$ be an
integer. Let $F\subset \mathbb{P}^4$ be a hypersurface of degree
$\mu+1$ containing $C(D)$,  consisting of $\mu+1$ sufficiently
general hyperplanes. Let $R$ be the residual curve to $C(D)$ in the
complete intersection of $F$ with $S$. Equipped with the reduced
structure, $R$ is a cone over $k$ distinct points of $\Sigma$, with
$k-1=5\mu+\epsilon$. In particular, $R$ is a (reducible) a.C.M.
curve of degree $k$ on $S$, and, if we denote by $R'$ the hyperplane
section of $R$ with the hyperplane $\mathbb P^3\subset\mathbb P^4$
containing $\Sigma$, we have $p_a(R)=\sum_{i=1}^{+\infty}
k-h_{R'}(i)$ (compare with (\ref{basic2.1}) and the first line
below). On the other hand, since $R'\subset \Sigma$, it is clear
that $h_{R'}(i)=h_{k,1}(i)$ for every $i\geq 1$. Hence, we have:
$$
p_a(R)=G_{k,1}=5\binom{\mu}{2}+\mu(1+\epsilon)+1+\binom{\epsilon}{4}.
$$
Now, let $m\gg 0$, and let $G\subset \mathbb{P}^4$ be a hypersurface
of degree $m+1$ containing $C(D)$ such that the residual curve $C$
in the complete intersection of $G$ with $S$, equipped with the
reduced structure, is an integral curve of degree $d=5m+\epsilon+1$,
with a singular point of multiplicity $k$ at the vertex $p$ of $S$,
and tangent cone at $p$ equal to $R$. We are going to prove that $C$
is the curve we are looking for, i.e.
$$p_a(C)=G_{d,1}.$$ To this aim, let $\widetilde{S}$ be the blowing-up of $S$ at the
vertex. By \cite[p. 374]{Hartshorne}, we know that $\widetilde{S}$
is the ruled surface $\mathbb{P}(\mathcal O_{\Sigma}\oplus\mathcal
O_{\Sigma}(-1))\to \Sigma$. Denote by $E$ the exceptional divisor,
by $f$ the line of the ruling, and by $L$ the pull-back of the
hyperplane section. We have $L^2=5$, $L\cdot f=1$, $f^2=0$, $L\equiv
E+5f$ and $K_{\widetilde{S}}\equiv -2L+5f$. Let
$\widetilde{C}\subset \widetilde{S}$ be the blowing-up of $C$ at
$p$, which is nothing but the normalization of $C$. Since $C$ has
degree $d$,  $\widetilde{C}$ belongs to the numerical class of
$(m+1+a)L+(1+\epsilon-5(a+1))f$ for some integer $a$. Moreover
$E\cdot \widetilde{C}=1+\epsilon-5(a+1)=k$, so
$$
a=-\frac{k+4-\epsilon}{5}=-\mu-1.
$$
By the adjunction formula we get
$$
p_a(\widetilde{C})=5\binom{m}{2}+m(\epsilon+1)+1-\frac{5}{2}a^2+a\left(\epsilon-\frac{3}{2}\right).
$$
On the other hand, we have $ p_a(C)=p_a(\widetilde{C})+\delta_p$,
where $\delta_p$ is the delta invariant of the singularity $(C,p)$.
Since the tangent cone of $C$ at $p$ is $R$, the delta invariant is
equal to the difference between the arithmetic genus of $R$ and the
arithmetic genus of $k$ disjoint lines in the projective space, i.e.
$$
\delta_p=p_a(R)-(1-k)=5\binom{\mu}{2}+\mu(1+\epsilon)+1+\binom{\epsilon}{4}-(1-k).
$$
It follows that
$$
p_a(C)=5\binom{m}{2}+m(\epsilon+1)+1-\frac{5}{2}a^2+a\left(\epsilon-\frac{3}{2}\right)
$$
$$
+5\binom{\mu}{2}+\mu(1+\epsilon)+1+\binom{\epsilon}{4}-(1-k).
$$
Taking into account that $a=-\mu-1$, a direct computation proves
that this number is exactly $G_{d,1}$. This concludes the proof of
Theorem \ref{bound}.

\bigskip
\section{The proof of Theorem \ref{boundP5}.}

By (\ref{basic4}), every surface of degree $\leq 5$ in $\mathbb P^5$
is contained in a quadric. Moreover,  if $C$ is not contained in a
surface of degree $<7$, and $d>179$,  then (see Section 2, $(vi)$,
and (\ref{b2})):
$$
p_a(C)\leq G(5;d,7)\leq \frac{d^2}{14}-\frac{3}{14}d+115.
$$
An elementary comparison, relying on (\ref{b1}), shows that this
number is strictly less than $G(7;d)=6\binom{m}{2}+m\epsilon$ for
$d>179$. Therefore, in order to prove Theorem \ref{boundP5}, we may
assume that $C$ is contained in a surface $S$ of degree $6$. Such a
surface is unique by Bezout's theorem, and is contained in a
hypersurface of degree $3$ by (\ref{basic4}). If $\pi$ denotes the
sectional genus of $S$, by Castelnuovo's bound we have $0\leq
\pi\leq 2$, and $\pi$ cannot be equal to $1$ or $2$, otherwise, by
(\ref{basic4}), $S$ is contained in a quadric. It follows that
$\pi=0$.

Let $S\subset \mathbb P^5$ be a surface of degree $6$ and sectional
genus $\pi=0$. By \cite[Corollary 3.6]{DGF} we know that $ -3\leq
p_a(S)\leq 0.$ Now we are going to prove that if $-3\leq p_a(S)\leq
-1$, then $S$ is contained in a quadric.

\smallskip
$\bullet$ If $p_a(S)=-3$, by \cite[Remark 3.7]{DGF} we get
$h^1(\mathbb P^5,\mathcal I_S(1))=0$. Therefore, the map
$H^0(\mathbb P^5,\mathcal I_S(2))\to H^0(\mathbb P^4,\mathcal
I_{\Sigma}(2))$ is onto. This implies that $S$ is contained in a
quadric, because it is so for every curve of degree $6$ in $\mathbb
P^4$ by (\ref{basic}) and (\ref{basic3}).

\smallskip
$\bullet$ Assume $p_a(S)=-2$. Since $\Sigma$ has degree $6$ and
arithmetic genus $0$, by \cite[Proposition 3.1]{DGF} we know that
$h^1(\mathbb P^4,\mathcal I_{\Sigma}(1))=2$, $h^1(\mathbb
P^4,\mathcal I_{\Sigma}(2))\leq 1$, and $h^1(\mathbb P^4,\mathcal
I_{\Sigma}(i))=0$ for $i\geq 3$.

In the case $h^1(\mathbb P^4,\mathcal I_{\Sigma}(2))=0$,  by
\cite[Lemma 3.4]{DGF} we get (compare with (\ref{mui})): $
-2=p_a(S)=-\dim M(\Sigma)
+\sum_{i=1}^{+\infty}\mu_i=-2+\sum_{i=1}^{+\infty}\mu_i.$ Therefore,
$\sum_{i=1}^{+\infty}\mu_i=0$, and so $h^1(\mathbb P^5,\mathcal
I_S(1))=0$. As before, we deduce that $S$ lies on a quadric.

In the case $h^1(\mathbb P^4,\mathcal I_{\Sigma}(2))=1$, we have
$h^0(\mathbb P^4,\mathcal I_{\Sigma}(2))=3$. In view of the exact
sequence $ H^0(\mathbb P^5,\mathcal I_{S}(2))\to H^0(\mathbb
P^4,\mathcal I_{\Sigma}(2))\to H^1(\mathbb P^5,\mathcal I_{S}(1)),$
in order to prove that $S$ lies on a quadric, it suffices to prove
that
\begin{equation}\label{triv}
h^1(\mathbb P^5,\mathcal I_{S}(1))\leq 2.
\end{equation}
This follows from the exact sequence $ 0\to\mathcal I_S\to \mathcal
O_{\mathbb P^5}\to \mathcal O_S\to 0$, taking into account that $
h^0(\mathcal O_S(1))\leq 8,$ in view of the exact sequence $
0\to\mathcal O_S\to \mathcal O_{S}(1)\to \mathcal O_{\Sigma}(1)\to
0.$

\smallskip
$\bullet$ Assume $p_a(S)=-1$. As before, we have two cases,
$h^1(\mathbb P^4,\mathcal I_{\Sigma}(2))=0$, or $h^1(\mathbb
P^4,\mathcal I_{\Sigma}(2))=1$.

If $h^1(\mathbb P^4,\mathcal I_{\Sigma}(2))=0$, then $
-1=p_a(S)=-\dim M(\Sigma)
+\sum_{i=1}^{+\infty}\mu_i=-2+\sum_{i=1}^{+\infty}\mu_i.$ Hence
$\sum_{i=1}^{+\infty}\mu_i=1$. Combining with the exact sequence
$
H^0(\mathbb P^5,\mathcal I_{S}(2))\to H^0(\mathbb P^4,\mathcal
I_{\Sigma}(2))\to H^1(\mathbb P^5,\mathcal I_{S}(1))\to H^1(\mathbb
P^5,\mathcal I_{S}(2)),
$
we get $H^0(\mathbb P^5,\mathcal I_{S}(2))\neq 0$, otherwise the map
on the right should have a kernel of dimension $\geq 2=h^0(\mathbb
P^4,\mathcal I_{\Sigma}(2))$, in contrast with the fact that
$\mu_2\leq 1$.

If $h^1(\mathbb P^4,\mathcal I_{\Sigma}(2))=1$, then $h^0(\mathbb
P^4,\mathcal I_{\Sigma}(2))=3$. Hence, the map $H^0(\mathbb
P^4,\mathcal I_{\Sigma}(2))\to H^1(\mathbb P^5,\mathcal I_{S}(1))$
has non trivial kernel, because $h^1(\mathbb P^5,\mathcal
I_{S}(1))\leq 2$ (see (\ref{triv})). This implies $H^0(\mathbb
P^5,\mathcal I_{S}(2))\neq 0$.

\smallskip
Summing up, in order to prove Theorem \ref{boundP5}, we may assume
that $C$ is contained in a surface $S$ of degree $6$, sectional
genus $\pi=0$, and arithmetic genus $p_a(S)=0$. On such a surface,
the genus of $C$ satisfies the bound in view of \cite[Corollary
3.11]{DGF} (here we need $d>215$). Moreover, $C$ cannot be a.C.M..
Otherwise, $M(S)=0$ (see Section 2, $(vii)$). In particular,
$H^1(\mathbb P^5,\mathcal I_S(1))=0$, and so the restriction map
$H^0(\mathbb P^5,\mathcal I_S(2))\to H^0(\mathbb P^{4},\mathcal
I_{\mathbb P^{4},\Sigma}(2))$ is onto. This is impossible, because
$\Sigma$ is contained in a quadric, and $S$ not.

\smallskip
It remains to prove that the bound is sharp. To this purpose, let
$S'\subset \mathbb P^7$ be a smooth rational normal scroll surface,
of degree $6$. Let $C'$ be a Castelnuovo's curve on $S'$ of degree
$d$. Let $S$ be the general projection of $S'$ in $\mathbb P^5$.
Then the image $C$ of $C'$ is an extremal curve. In fact, $S$ is not
contained in quadrics. To prove this, we use the fact that $S$ is of
maximal rank \cite[Theorem 2, Lemma 3.1]{BE}. In particular, the map
$ H^0(\mathbb P^5,\mathcal O_{\mathbb P^5}(2))\to H^0(S,\mathcal
O_{S}(2))$ is injective or surjective. Since both spaces have the
same dimension (see (\ref{scroll})), it follows that $h^0(\mathbb
P^5,\mathcal I_S(2))=0$. This concludes the proof of Theorem
\ref{boundP5}.

\bigskip
\section{The proof of Proposition \ref{qboundPr}.}

By the definition of $s$, we have $ \binom{r+2}{2}-(3+3(s-1))>0.$
Hence, by (\ref{basic4}), every surface of degree $<s$ is contained
in a quadric. In particular, the extremal curve $C$ cannot be
contained in a surface of degree $<s$. Moreover, if $C$ is not
contained in a surface of degree $<s+1$, then (compare with
(\ref{es2}), (\ref{es1}), and (\ref{b1}))
\begin{equation}\label{pq}
p_a(C)\leq G(r;d,s+1)<
G(s+1;d)=\frac{d^2}{2s}-\frac{d}{2s}(s+2)+\frac{1+\epsilon}{2s}(s+1-\epsilon).
\end{equation}

\medskip
On the other hand, let $S'\subset \mathbb P^{s+1}$ be a smooth
rational normal scroll surface, of degree $s$. Let $D'$ be a
Castelnuovo's curve on $S'$ of degree $d$. Let $S\,(\cong S')$ be
the general projection of $S'$ in $\mathbb P^r$. By \cite[loc.
cit.]{BE}, $S$ is of maximal rank. Since
$$
h^0(\mathbb P^r, \mathcal O_{\mathbb P^r}(2))=\binom{r+2}{2}\leq
3+3s=h^0(S, \mathcal O_{S}(2))
$$
(compare with the definition of $s$ and (\ref{scroll})), it follows
that the restriction map $ H^0(\mathbb P^r, \mathcal O_{\mathbb
P^r}(2))\to H^0(S, \mathcal O_{S}(2))$ is injective. Hence,  $S$,
and so the image $D$ of $D'$, are not contained in quadrics.
Moreover, since $D$ is a Castelnuovo's curve, we have (compare with
(\ref{b1}))
\begin{equation}\label{pq1}
p_a(D)=p_a(D')=G(s+1;d)>\frac{d^2}{2s}-\frac{d}{2s}(s+2).
\end{equation}

\medskip
In view of (\ref{pq}) and (\ref{pq1}), in order to prove
(\ref{bqs}), we may assume that $C$ is contained in a surface $S$ of
degree $s$. By (\ref{basic4}), the sectional genus of $S$ should be
$\pi=0$, otherwise $S$ is contained in a quadric (here we have to
assume $k\neq 1$, i.e. that $r$ is not divisible by $3$). Now, if
$C$ is contained in a surface $S$ of degree $s$ and sectional genus
$\pi=0$, then by \cite[Lemma]{PAMS} we know that
$$
p_a(C)\leq \frac{d^2}{2s}-\frac{d}{2s}(s+2)+O(1),
\quad{\text{with}}\quad O(1)\leq \frac{s^3}{r-2}.
$$
Taking into account (\ref{pq1}), we deduce (\ref{bqs}).

\medskip
Since $d>d_1(r,s)$, by Bezout's theorem  the surface $S$ containing
$C$ is unique. By (\ref{basic4}) and the definition of $s$, it
follows that $S$ is contained in a cubic hypersurface. Moreover, $C$
cannot be a.C.M.. Otherwise, $M(S)=0$ (see Section 2, $(vii)$). In
particular, $H^1(\mathbb P^r,\mathcal I_S(1))=0$, and so the
restriction map $H^0(\mathbb P^r,\mathcal I_S(2))\to H^0(\mathbb
P^{r-1},\mathcal I_{\mathbb P^{r-1},\Sigma}(2))$ is onto
($\Sigma\,=$ the general hyperplane section of $S$). This is
impossible, because, in view of (\ref{basic}) and (\ref{basic3}),
$\Sigma$ is  contained in a quadric, and $S$ not. This concludes the
proof of Proposition \ref{qboundPr}.

\begin{remark}\label{k=1}
When $3$ divides $r$, i.e. when $k=1$, previous argument does not
work, because it may happen that, for $\pi>0$, a surface of degree
$s$ ($3s+1=\binom{r+2}{2}$) is not contained in quadrics (however,
it is necessary that $\pi\leq 1$ by (\ref{basic4})). For instance,
one may consider a general projection $S$ in $\mathbb P^6$ of a
$3$-ple Veronese embedding of $\mathbb P^2$ in $\mathbb P^9$. In
this case, repeating the same argument as before (compare with
\cite[Theorem 3]{BE} and \cite[loc. cit.]{PAMS}), one may prove that
the sharp bound (for curves in $\mathbb P^6$ not contained in
quadrics, and of degree $d=0$ mod. $3$, $d>d_1(r,s)$) is
\begin{equation}\label{k=1.1}
p_a(C)=\frac{d^2}{2s}-\frac{d}{2}+O(1),
\end{equation}
with $s=9$, and $0<O(1)\leq 182$. The formula (\ref{k=1.1})  remains
true (with $0<O(1)\leq \frac{s^3}{r-2}$) if there exists a surface
of degree $s$ in $\mathbb P^r$, with sectional genus $\pi=1$, and
not contained in quadrics (at least when $\epsilon=s-1$, i.e. when
$d$ is a multiple of $s$).
\end{remark}

\bigskip
\section{The proof of Proposition \ref{cboundPr}.}

The proof is quite similar to the proof of previous Proposition
\ref{qboundPr}. Hence, we omit some details.

\medskip
By the definition of $s$ and  by (\ref{basic4}), every surface of
degree $<s$ is contained in a cubic (here we need that
$6s+4\leq\binom{r+3}{3}$). In particular, the extremal curve $C$
cannot be contained in a surface of degree $<s$. Moreover, if $C$ is
not contained in a surface of degree $<s+1$, then (compare with
(\ref{es2}), (\ref{es1}), and (\ref{b1}))
\begin{equation}\label{.pq}
p_a(C)\leq G(r;d,s+1)<
G(s+1;d)=\frac{d^2}{2s}-\frac{d}{2s}(s+2)+\frac{1+\epsilon}{2s}(s+1-\epsilon).
\end{equation}

\medskip
On the other hand, let $S'\subset \mathbb P^{s+1}$ be a smooth
rational normal scroll surface, of degree $s$. Let $D'$ be a
Castelnuovo's curve on $S'$ of degree $d$. Let $S\,(\cong S')$ be
the general projection of $S'$ in $\mathbb P^r$. By \cite[loc.
cit.]{BE}, $S$ is of maximal rank. Since
$$
h^0(\mathbb P^r, \mathcal O_{\mathbb P^r}(3))=\binom{r+3}{3}=
6s+4=h^0(S, \mathcal O_{S}(3))
$$
(compare with the definition of $s$ and (\ref{scroll})), it follows
that the restriction map $ H^0(\mathbb P^r, \mathcal O_{\mathbb
P^r}(3))\to H^0(S, \mathcal O_{S}(3))$ is injective (actually, here
we only need that $6s+4\geq\binom{r+3}{3}$). Hence, $S$, and so the
image $D$ of $D'$, are not contained in  quadrics. Moreover, since
$D$ is a Castelnuovo's curve, we have (compare with (\ref{b1}))
\begin{equation}\label{.pq1}
p_a(D)=p_a(D')=G(s+1;d)>\frac{d^2}{2s}-\frac{d}{2s}(s+2).
\end{equation}

\medskip
In view of (\ref{.pq}) and (\ref{.pq1}), in order to prove
(\ref{bqs1}), we may assume that $C$ is contained in a surface $S$
of degree $s$. By (\ref{basic4}), the sectional genus of $S$ should
be $\pi=0$, otherwise $S$ is contained in a cubic (again, here we
need that $6s+4\leq\binom{r+3}{3}$). Now, if $C$ is contained in a
surface $S$ of degree $s$ and sectional genus $\pi=0$, then by
\cite[Lemma]{PAMS} we know that
$$
p_a(C)\leq \frac{d^2}{2s}-\frac{d}{2s}(s+2)+O(1),
\quad{\text{with}}\quad O(1)\leq \frac{s^3}{r-2}.
$$
Taking into account (\ref{.pq1}), we deduce (\ref{bqs1}). One may
prove the remaining properties in a similar way as in Proposition
\ref{qboundPr}. This concludes the proof of Proposition
\ref{cboundPr}.

\bigskip
\section{Appendix.}

In order to conclude the proof of Theorem \ref{bound}, first part,
it remains  to prove that {\it if $C$ is not contained in a surface
of degree $<5$, then $p_a(C)<\frac{1}{8}d(d-6)+1$}, when $16 < d
\leq 143$. We are going to do an analysis similar to the one that
appears in \cite[pp. 74-75]{DGFr}. The same calculation in
\cite[loc. cit.]{DGFr} proves that $p_a(C)<\frac{1}{8}d(d-6)+1$ for
$d> 30$. We need to refine the argument to deal with the case $16 <d
\leq 30$. Let $\Gamma\subset \mathbb P^3$ be the general hyperplane
section of $C$. In the sequel, we will apply  (\ref{basic2}),
(\ref{basic2.1}), (\ref{basic2.2}), and (\ref{basic2.3}) (see
Section $2$, $(i)$).

\bigskip $\bullet$ Case I: $h^0(\mathbb P^3,\mathcal
I_{\Gamma}(2))\geq 2$.

This case can't happen. Otherwise, since $d> 4$, by monodromy
\cite[Proposition 2.1]{CCD}, $\Gamma $ would be contained in an
integral curve of $\mathbb P^3$ of degree $\leq 4$. Since $d>16$,
from \cite[Theorem (0.2)]{CC} we would deduce that $C$ is contained
in a surface of degree $\leq 4 $, against our hypothesis.

\bigskip $\bullet$ Case II: $h^0(\mathbb P^3,\mathcal
I_{\Gamma}(2))=1$ and $h^0(\mathbb P^3,\mathcal I_{\Gamma}(3))>4$.

This is the most complicated case. Since $d> 6$, by monodromy
\cite[Proposition 2.1]{CCD}, $\Gamma$ is contained in an integral
curve $ X $ of $ \mathbb P^3 $ of degree $ \leq 6$. Based on what
was said in the previous case, we may suppose the degree of $X$ is $
5 $ or $ 6 $.

{\it First assume $\deg X=5$.}

If $d\geq 21$, by Bezout's theorem we have $h_{\Gamma}(i)=h_X(i)$
for all $i\leq 4$. Let $X'$ be the general hyperplane section of
$X$. By (\ref{basic2}) we get $h_X(i)\geq \sum_{j=0}^{i}h_{X'}(j)$.
By (\ref{basic2.3}) it follows that $h_X(3)\geq 14$ and $h_X(4)\geq
19$. Therefore, if $d\geq 21$, taking into account (\ref{basic2.2}),
we get:
$$
h_{\Gamma}(1)=4,\, h_{\Gamma}(2)=9,\, h_{\Gamma}(3)\geq 14,\,
h_{\Gamma}(4)\geq 19,
$$
$$
h_{\Gamma}(5)\geq \min\{d,\, 22\}, \, h_{\Gamma}(6)\geq \min\{d,\,
27\},\, h_{\Gamma}(7)=d.
$$

If $28\leq d\leq 30$, using (\ref{basic2.1}) we deduce:
$$
p_a(C)\leq \sum_{i=1}^{+\infty}d-h_{\Gamma}(i)\leq
(d-4)+(d-9)+(d-14)+(d-19)+8+3=4d-35,
$$
which is $<\frac{1}{8}d(d-6)+1$.

If $24\leq d\leq 27$, we have:
$$
p_a(C)\leq \sum_{i=1}^{+\infty}d-h_{\Gamma}(i)\leq
(d-4)+(d-9)+(d-14)+(d-19)+5=4d-41,
$$
which is $<\frac{1}{8}d(d-6)+1$, except for $d=24$ for which we have
$4d-41=\frac{1}{8}d(d-6)+1$. However, $p_a(C)$ should be strictly
less than such number, otherwise $p_a(C)=
\sum_{i=1}^{+\infty}d-h_{\Gamma}(i)$ and $C$ would be a.C.M.
(Section 2, $(i)$). This is impossible, because $C$ is not contained
in quadrics, while $h^0(\mathbb P^3,\mathcal I_{\Gamma}(2))=1$ (we
will use this argument later as well).

If $21\leq d\leq 23$, we have:
$$
p_a(C)\leq \sum_{i=1}^{+\infty}d-h_{\Gamma}(i)\leq
(d-4)+(d-9)+(d-14)+(d-19)+1=4d-45,
$$
which is $<\frac{1}{8}d(d-6)+1$.

Now assume $17\leq d\leq 20$. By Bezout's theorem we have
$h_{\Gamma}(3)=h_X(3)$. Hence we get:
$$
h_{\Gamma}(1)=4,\, h_{\Gamma}(2)=9,\, h_{\Gamma}(3)\geq 14,\,
h_{\Gamma}(4)\geq 17, \, h_{\Gamma}(7)=d.
$$
And we may conclude with a calculation similar to the previous one
(when $ 18 \leq d \leq 20 $, we use again the fact that $ C $ cannot
be a.C.M.).

{\it Assume  $\deg X=6$.}

In this case $X$ is a complete intersection of bi-degree $(2,3)$.

If $d\geq 25$, by Bezout's theorem we have  $h_{\Gamma}(i)=h_X(i)$
for every $i\leq 4$. Let $X'$ be the general plane section of $X$.
Similarly as before, since $\deg X=6$, we have $h_X(3)\geq 15$ e
$h_X(4)\geq 21$. Therefore, if $d\geq 25$, we have:
$$
h_{\Gamma}(1)=4,\, h_{\Gamma}(2)=9,\, h_{\Gamma}(3)\geq 15,\,
h_{\Gamma}(4)\geq 21,
$$
$$
h_{\Gamma}(5)\geq \min\{d,\, 23\}, \, h_{\Gamma}(6)\geq \min\{d,\,
29\},\, h_{\Gamma}(7)=d.
$$
It follows that:
$$
p_a(C)\leq \sum_{i=1}^{+\infty}d-h_{\Gamma}(i)\leq
(d-4)+(d-9)+(d-15)+(d-21)+7=4d-42,
$$
which is $<\frac{1}{8}d(d-6)+1$.

If $19\leq d\leq 24$ we have $h_{\Gamma}(3)=h_X(3)$, and so:
$$
h_{\Gamma}(1)=4,\, h_{\Gamma}(2)=9,\, h_{\Gamma}(3)\geq 15,\,
h_{\Gamma}(4)\geq 18, h_{\Gamma}(5)\geq \min\{d,\, 23\}, \,
h_{\Gamma}(6)=d.
$$
Hence:
$$
p_a(C)\leq \sum_{i=1}^{+\infty}d-h_{\Gamma}(i)\leq
(d-4)+(d-9)+(d-15)+(d-18)+1=4d-45,
$$
which is $<\frac{1}{8}d(d-6)+1$. It remains to examine the cases
$d=18$ e $d=19$.

If $d=18$ and $h_{\Gamma}(3)=h_X(3)$, then $h_{\Gamma}(3)\geq 15$.
Otherwise, $\Gamma$ is a complete intersection of type $(2,3,3)$,
and from Koszul's complex we get $h_{\Gamma}(3)\geq 14$. Therefore,
in every case, we have:
$$
h_{\Gamma}(1)=4,\, h_{\Gamma}(2)=9,\, h_{\Gamma}(3)\geq 14,\,
h_{\Gamma}(4)\geq 17,\, h_{\Gamma}(5)=18.
$$
Hence $p_a(C)< \sum_{i=1}^{+\infty}18-h_{\Gamma}(i)\leq
28=\frac{1}{8}18(18-6)+1$.

When $d=17$, previous argument does not work. We may argue as
follows. If $h_{\Gamma}(3)=h_{X}(3)$ we may repeat previous
computation (in this case $h_X(3)\geq 15$). Otherwise,  $h^0(\mathbb
P^3,\mathcal I_{\Gamma}(3))>h^0(\mathbb P^3,\mathcal I_{X}(3))$. On
the other hand, since $X$ is a complete intersection of bi-degree
$(2,3)$, we have $h^1(\mathbb P^3,\mathcal I_{X}(2+i))=h^2(\mathbb
P^3,\mathcal I_{X}(2+i))=0$ for all $i\geq 0$. It follows that
\cite[Corollary 1.2]{CCD}
$$
\Delta h_{\Gamma}(4)\leq \max\left\{0,\, \Delta
h_{\Gamma}(3)-2\right\}.
$$
If $\Delta h_{\Gamma}(4)\leq 0$ then $h_{\Gamma}(3)=17$. Otherwise,
$\Delta h_{\Gamma}(4)\leq \Delta h_{\Gamma}(3)-2$, hence
$h_{\Gamma}(3)\geq 14$, and we may proceed as in  previous
computation.

\bigskip $\bullet$ Case III : $h^0(\mathbb P^3,\mathcal
I_{\Gamma}(2))=1$ e $h^0(\mathbb P^3,\mathcal I_{\Gamma}(3))=4$.

We have:
$$
h_{\Gamma}(1)=4,\, h_{\Gamma}(2)=9,\, h_{\Gamma}(3)=16, \,
h_{\Gamma}(4)\geq \min\{d,\, 19\},
$$
$$
h_{\Gamma}(5)\geq \min\{d,\, 24\}, \, h_{\Gamma}(6)=d.
$$

If $26\leq d\leq 30$, we have:
$$
p_a(C)< \sum_{i=1}^{+\infty}d-h_{\Gamma}(i)\leq
(d-4)+(d-9)+(d-16)+11+6=3d-12,
$$
which is $\leq \frac{1}{8}d(d-6)+1$.

If $21\leq d\leq 25$, we have:
$$
p_a(C)< \sum_{i=1}^{+\infty}d-h_{\Gamma}(i)\leq
(d-4)+(d-9)+(d-16)+6+1=3d-22,
$$
which is $\leq \frac{1}{8}d(d-6)+1$, except the case $d=21$. In this
case $3d-22=41$, hence $p_a(C)\leq 40<
\frac{1}{8}21(21-6)+1=40+\frac{3}{8}$.

If $17\leq d\leq 20$, we have:
$$
p_a(C)< \sum_{i=1}^{+\infty}d-h_{\Gamma}(i)\leq
(d-4)+(d-9)+(d-16)+1=3d-28,
$$
which is $\leq \frac{1}{8}d(d-6)+1$.

\bigskip $\bullet$ Case IV: $h^0(\mathbb P^3,\mathcal
I_{\Gamma}(2))=0$.

We have:
$$
h_{\Gamma}(1)=4,\, h_{\Gamma}(2)=10,\, h_{\Gamma}(3)\geq 13, \,
h_{\Gamma}(4)\geq \min\{d,\, 19\},
$$
$$
h_{\Gamma}(5)\geq \min\{d,\, 22\},\, h_{\Gamma}(6)\geq \min\{d,\,
28\}, \, h_{\Gamma}(6)=d.
$$
So, we may conclude with a similar computation as in the previous
case.

\bigskip
\begin{remark}
From previous analysis it follows that {\it if $C\subset\mathbb P^4$
is not contained in quadrics, and $d>16$, then $h^0(\mathbb
P^3,\mathcal I_{\Gamma}(2))\leq 1$}. In fact, otherwise $C$ should
be contained  in an isomorphic projection of the Veronese surface.
Therefore, $\Gamma$ should be contained in a rational quartic $X$ in
$\mathbb P^3$. And $h_{\Gamma}(2)=h_X(2)=9$ (in fact $h^1(\mathbb
P^3,\mathcal I_X(2))=0$).
\end{remark}

\end{document}